\documentclass{amsart}
\usepackage{latexsym,amssymb,amsthm,amsmath}

\usepackage{amssymb}
\usepackage{amsmath}

\theoremstyle{plain}
\newtheorem{theorem}{Theorem}[section]
\newtheorem*{Theorem B}{Theorem B}
\newtheorem*{Theorem A}{Theorem A}

\newtheorem{proposition}{Proposition}[section]
\newtheorem{corollary}{Corollary}[section]

\newtheorem{definition}{Definition}[section]

\numberwithin{equation}{section}

\theoremstyle{remark}

\begin{document}
\title[invariant and anti-invariant submanifolds of special quasi-sasakian manifolds]
{invariant  and anti-invariant submanifolds of special quasi-sasakian manifolds}
\author[S. K. Hui and J. Roy]{Shyamal Kumar Hui and Joydeb Roy}
\begin{abstract}
The present paper deals with the study of Chaki-pseudo parallel and Deszcz-pseudo parallel invariant submanifolds of
SQ-Sasakian manifolds with respect to Levi-Civita connection and semisymmetric metric connection and obtain that these two
 classes are equivalent with a certain condition. Also the invariant and anti-invariant submanifolds of SQ-Sasakian manifolds with respect to
 Levi-Civita connection as well as semisymmetric  metric connection whose metrics are Ricci solitons are studied.
\end{abstract}
\maketitle
\section{Introduction}
A $(2n+1)$-dimensional smooth manifold $\bar{M}$ is said to be an almost contact metric
structure $(\phi,\xi,\eta,g)$ if it satisfies the following relations \cite{BLAIR}
\begin{equation}\label{eqn1.1}
\phi^2 X= -X +\eta(X)\xi,\ \  \phi\xi=0,
\end{equation}
\begin{equation}\label{eqn1.2}
\eta (\xi)=1,\  g(X,\xi)= \eta(X),\ \eta(\phi X)=0,
\end{equation}
\begin{equation}\label{eqn1.3}
g(\phi X,\phi Y)= g(X,Y)- \eta (X)\eta(Y),\ g(\phi X,Y)=-g(X,\phi Y)
\end{equation}
for any vector fields X and Y on $\bar{M}$, where $\phi$ is a tensor of type $(1,1)$, $\xi$ is a vector field, $\eta$ is an $1$-form
and $g$ is a Riemannian metric on $\bar{M}$. A manifold equipped with an almost contact metric structure $(\phi,\xi,\eta,g)$
is called an almost contact metric manifold.

The fundamental $2$-form $\Phi$ on $\bar{M}$ is defined by $\Phi(X,Y)= g(X,\phi Y)$ and the
Nijenhuis tesor of $\bar{M}$ is given by
\begin{equation*}
  N(X,Y)= [\phi X,\phi Y]+ \phi^2[X,Y]-\phi[\phi X,Y]-\phi[X,\phi Y] + 2d\eta(X,Y)\xi,
\end{equation*}
where $d$ denotes the exterior derivative. Also the almost contact metric structure $(\phi,\xi,\eta)$ is called normal if and only if the Nijenhuis tensor vanishes. An almost contact metric manifold $\bar{M}^{2n+1}(\phi,\xi,\eta,g)$ is said to be \cite{Blair} \\(i) Sasakian if $\Phi= d\eta$ and
$(\phi,\xi,\eta)$ is normal,\\(ii) cosymplectic if $\Phi$ and $\eta$ are closed (i.e. $d\Phi=0$ and $d\eta=0$) and $(\phi,\xi,\eta)$ is normal and \\ $(iii)$ quasi-Sasakian if $\Phi$ is closed (i.e.,$d\Phi=0$)and $(\phi,\eta,\xi)$ is normal.

A 3-dimensional almost contact metric manifold is quasi-sasakian \cite{OLSZ} if and only if
\begin{equation}\label{eqn1.4}
\bar{\nabla}_X\xi= -\beta \phi X
\end{equation}

for  a certain smooth function $\beta$ on $M$ such that $(\xi \beta)=0$, $\bar{\nabla}$ being the Levi-Civita connection on $\bar{M}$.
However, in a quasi-Sasakian manifold of dimension greater than three, the relation $(1.4)$ does not hold in general.
In $1993$, Kwon and Kim \cite{KWON} constructed an example of almost contact metric manifolds of dimension
 greater than three such that the following relations hold:
 \begin{equation}\label{eqn1.5}
 d\Phi= 0,\ \bar{\nabla}_X\xi= -\beta\phi X
 \end{equation}
and $(\phi,\xi,\eta)$ is normal for a certain smooth function $\beta$ on $M$ such that $(\xi\beta)=0$.
 This new class of almost contact metric manifolds satisfying $(1.5)$ is said to be a
 special quasi-Sasakian manifold (briefly, SQ-Sasakian manifold)\cite{KWON} and such a manifold is denoted by $M^*$.
 Here the smooth function $\beta$ is said to be the structure function of the
 SQ-Sasakian manifold. Also it may be noted that an
 SQ-Sasakian manifold is cosymplectic if and only if $\beta=0$ and Sasakian if and only if $\beta=1$.
 Recently Shaikh and Ahmad \cite{SHAIKH} studied SQ-Sasakian manifolds.

 Due to significant applications in applied mathematics and theoretical Physics,
 the geometry of submanifolds has become an important subject.
 There are so many classes of submanifolds. In \cite{BEJ} Bejancu and Papaghuic
 introduced the notion of invariant submanifolds that geometry inherits
 almost all properties of ambient manifold. Thereafter several authors studied invariant submanifolds of different ambient manifolds.
 The present paper deals with the study of invariant submanifolds of SQ-Sasakian manifolds $M^*$. After introduction in section $1$,
 section $2$ deals with some preliminaries.

 As a generalisation of Ricci symmetric manifolds, Chaki \cite{CHAKI} introduced the notion of pseudo Ricci symmetric manifolds.
 We may say it as Chaki-pseudo Ricci symmetric manifolds. A non-flat Riemannian manifold $(M,g)$ is said to be Chaki-pseudo
  Ricci symmetric \cite{CHAKI} if its Ricci tensor S of type $(0,2)$ is not identically zero and satisfies the condition
 \begin{equation}\label{eqn1.6}
 (\nabla_{X}S)(Y,Z)= 2\alpha(X)S(Y,Z)+\alpha(Y)S(X,Z)+\alpha(Z)S(X,Y)
 \end{equation}
for all vector fields $X,Y,Z \in \chi(M)$, where $\alpha$ is a nowhere vanishing $1$-form.

In another direction, as a generalization of Ricci semisymmetric manifolds \cite{SZABO}, Deszcz introduced the notion
of Ricci pseudo symmetric manifolds \cite{DESZCZ}. We may call it as Deszcz-pseudo Ricci symmetric manifolds.
A Riemannian manifold $(M,g)$ $(n>2)$ is said to be Deszcz-pseudo Ricci symmetric \cite{DESZCZ} if
\begin{equation}\label{eqn1.7}
(R(X,Y)\cdot S)(Z,U)= L_{S}Q(g,S)(Z,U;X,Y)
\end{equation}
holds on $U_{S}=\{x \in \ M : (S-\frac{r}{n}g)_{x}\neq 0\}$ for all $X,Y,Z,U \ \in \ \chi(M)$, where $L_S$ is
some function on $U_S$, $R$ is the curvature tensor, $S$ is the Ricci tensor and $r$ is the scalar curvature
of the manifold $M$ and
\begin{eqnarray*}
Q(g,S)(Z,U;X,Y)&=& g(Y,Z)S(X,U)-g(X,Z)S(Y,U)\\
\nonumber &+& g(Y,U)S(X,Z)-g(X,U)S(Y,Z).
\end{eqnarray*}
The similar concept also took place in the submanifold theory, where second fundamental form
$h$ and the induced metric tensor $g$ plays an important role. Recently Hui and Mandal \cite{HUI}
studied Chaki-pseudo parallel and Deszcz-pseudo parallel Contact CR-submanifolds of Kenmotsu manifolds.
Following the same, section $3$ is denoted to the study of Chaki-pseudo
parallel and Deszcz-pseudo parallel invariant submanifolds of SQ-Sasakian manifolds and proved
that these two classes are equivalent with a certain condition.
However, it may be noted that Chaki-pseudo Ricci symmetric manifolds is different from Deszcz-pseudo
Ricci symmetric manifolds. Also it is shown that Chaki-pseudo parallel invariant submanifolds and
Deszcz-pseudo parallel invariant submanifolds of SQ-Sasakian manifolds admitting semisymmetric
metric connection are equivalent with a certain condition.

A Ricci soliton $(g,V,\lambda)$ on a Riemannian manifold $(M,g)$ is a generalization of an Einstein
metric such that \cite{HAMILTON}
\begin{equation}\label{eqn1.8}
\pounds_{V}g+2S+2\lambda g=0,
\end{equation}
where $S$ is the Ricci tensor, $\pounds_V$ is the Lie derivative operator along the vector field $V$
on $M$ and $\lambda$ is a real number. The Ricci soliton is shriking, steady and expanding according as $\lambda$
is negative, zero and positive respectively. In \cite{SHARMA} Sharma studied Ricci solitons in contact metric geometry.
Thereafter many authors studied Ricci solitons in different contact metric manifolds. Recently Hui et al. \cite{Hui} Ricci
solitons on submanifolds of $(LCS)_n$-manifolds. In section $4$ of the paper, we study Ricci solitons on invariant submanifolds
of SQ-Sasakian manifolds. Section $5$ is devoted to the study of anti-invariant submanifolds of SQ-Sasakian manifolds and Ricci 
solitons. Finally, we conclude in section $6$.
\section{Preliminaries}
 In an SQ-Sasakian manifold $M^*$, the following relations hold \cite{SHAIKH}
 \begin{equation}\label{eqn2.1}
 (\bar{\nabla}_{X} \phi)(Y)= \beta[g(X,Y)\xi-\eta(Y)X],
 \end{equation}
\begin{equation}\label{eqn2.2}
(\bar{\nabla}_X \eta)(Y)= \beta g(X,\phi Y),
\end{equation}
\begin{equation}\label{eqn2.3}
\bar{R}(X,Y)\xi= (Y\beta)\phi X-(X\beta)\phi Y+\beta^2[\eta(Y)X-\eta(X)Y],
\end{equation}
\begin{eqnarray}
\label{eqn2.4}
\eta(\bar{R}(X,Y)Z)&=&(X\beta)g(\phi Y,Z)- (Y\beta)g(\phi X,Z)\\
\nonumber &+& \beta^2[g(Y,Z)\eta(X) - g(X,Z)\eta(Y)],
\end{eqnarray}
\begin{equation}\label{eqn2.5}
  \bar{R}(\xi,X)Y= g(X,\phi Y)grad\beta +(Y\beta)\phi X +\beta^2[g(X,Y)\xi -\eta(Y)X],
\end{equation}
\begin{equation}\label{eqn2.6}
  \bar{S}(Y,\xi)= 2\eta\beta^2\eta(Y)-((\phi Y)\beta)
\end{equation}
for all vector fields $X,Y,Z$ on $M^*$ and $\bar{R}$ is the curvature tensor and $\bar{S}$ is the Ricci tensor of $M^*$

Let $M$ be a $(2m+1)$ dimensional $(m<n)$ submanifold of a SQ-Sasakian manifold $M^*$. Let us take $\nabla$
 and $\nabla^\bot$ be the induced connections on the tangent bundle $TM$ and the normal bundle $T^\bot M$ of $M$
  respectively. Then the Gauss and Weingarten formulae are given by
\begin{equation}\label{eqn2.7}
\bar{\nabla}_{X}Y=\nabla_{X}Y + h(X,Y)
\end{equation}
and
\begin{equation}\label{eqn2.8}
\bar{\nabla}_{X}V=-A_{V}X+ \nabla^{\bot}_{X}V,
\end{equation}
where $h$ and $A_V$ are second fundamental form and shape operator and they are related by $g(h(X,Y),V)=g(A_{V}X,Y)$
for all $X,Y  \in \Gamma(TM)$ and $V \in \Gamma(T^\bot M)$. If $h=0$ then
the submanifold is said to be totally geodesic. The covariant derivative of $h$ is given by
\begin{equation}\label{eqn2.9}
(\nabla_{X}h)(Y,Z)= \nabla^\bot_X(h(Y,Z))- h(\nabla_{X}Y,Z)-h(Y,\nabla_{X}Z)
\end{equation}
for any vector fields $X,Y,Z$ tangent to $M$.

For any submanifold $M$ of a Riemannian manifold $M^*$, the equation of Gauss is given by
\begin{eqnarray}
\label{eqn2.10}
\bar{R}(X,Y)Z&=& R(X,Y)Z + A_{h(X,Z)}Y-A_{h(Y,Z)}X\\
\nonumber &+& (\bar{\nabla}_{X}h)(Y,Z)- (\bar{\nabla}_{Y}h)(X,Z)
\end{eqnarray}
for any $X,Y,Z \in \Gamma(TM)$, where $\bar{R}$ and $R$ denote the Riemannian curvature tensors of $M^*$ and $M$
respectively. In \cite{FRIED} Friedmann and Schouten  introduced the notion of semisymmetric linear connection on a
 smooth manifold. Then Hayden \cite{HAYDEN} introduced the idea of metric connection with torsion on a Riemannian manifold.
 A systematic study of the semisymmetric metric connection on a Riemannian maninfold has been given in \cite{YANO}. A linear
 connection on a SQ-Sasakian manifold $M^*$ is said to be a semisymmetric connection if its torsion tensor $\tau$ of the
 connection $\widetilde{\bar{\nabla}}$ is of the form
 \begin{equation}\label{eqn2.11}
 \tau(X,Y)= \widetilde{\bar{\nabla}}_{X}Y-\widetilde{\bar{\nabla}}_{Y}X-[X,Y]
 \end{equation}
 satisfies $\tau(X,Y)= \eta(Y)X-\eta(X)Y$, where $\eta$ is an $1$-form. Further, if the semi symmetric connection
 $\widetilde{\bar{\nabla}}$ satisfies the condition $(\widetilde{\bar{\nabla}}_{X}g)(Y,Z)=0$ for all
 $X,Y,Z\in \chi(\bar{M})$, Lie algebra of vector fields on $M^*$, then $\widetilde{\bar{\nabla}}$ is said to be
 semisymmetric metric connection. The relation between semisymmetric metric connection $\widetilde{\bar{\nabla}}$
  and Levi-Civita connection on SQ-Sasakian manifold $M^*$ is
\begin{equation}\label{eqn2.12}
\widetilde{\bar{\nabla}}_{X}Y= \bar{\nabla}_{X}Y+ \eta(Y)X-g(X,Y)\xi.
\end{equation}
If $\bar{R}$ and $\widetilde{\bar{R}}$ are respectively the curvature tensor with respect the Levi-Civita connection
$\bar{\nabla}$ and semisymmetric metric connection $\widetilde{\bar{\nabla}}$ in a SQ-Sasakian manifold then by
virtue of (\ref{eqn1.4}) and (\ref{eqn2.2}), we have from (\ref{eqn2.12}) that
\begin{eqnarray}
\label{eqn2.13}
\ \ \ \ \ \ \ \ \ \widetilde{\bar{R}}(X,Y)Z&=& \bar{R}(X,Y)Z + g(X,Z)Y- g(Y,Z)X\\
\nonumber &+&\eta(Z)\{\eta(Y)X-\eta(X)Y\}
+\{g(Y,Z)\eta(X)-g(X,Z)\eta(Y)\}\xi\\
\nonumber &+&\beta\{g(Y,Z)\phi X-g(X,Z)\phi Y+\Phi(X,Z)Y-\Phi(Y,Z)X\}.
\end{eqnarray}
also from (\ref{eqn2.13}) we obtain
\begin{eqnarray}
\label{eqn2.14}
\widetilde{\bar{S}}(Y,Z)&=& \bar{S}(Y,Z)
-(2n-1)\{g(Y,Z)-\eta(Y)\eta(Z)+\Phi(Y,Z)\},
\end{eqnarray}
where $\widetilde{\bar{S}}$ and $\bar{S}$ are respectively the Ricci tensor of a SQ-Sasakian manifold with respect to
semisymmetric metric connection $\widetilde{\bar{\nabla}}$ and Levi-Civita connection $\bar{\nabla}$.

Again from (\ref{eqn2.13}) we get
\begin{eqnarray}
\label{eqn2.15}
\widetilde{\bar{R}}(X,Y)\xi &=&\beta^2\{\eta(Y)X-\eta(X)Y\}+\{\beta\eta(Y)+(Y\beta)\}\phi X\\
\nonumber &-&\{\beta\eta(X)+ (X\beta)\}\phi Y,
\end{eqnarray}
and
\begin{eqnarray}
\label{eqn2.16}
\widetilde{\bar{R}}(\xi,Y)Z &=& \beta^2\{g(Y,Z)\xi-\eta(Z)Y\}+\{Z\beta-\beta\eta(Z)\}\phi Y\\
\nonumber &+& \Phi(Y,Z)\{grad\beta-\beta\xi\}
\end{eqnarray}
for arbitrary vector fields $X,Y$ and $Z$ on $M^*$.

As a generalization of quasi-Einstein (or $\eta$-Einstein) manifolds, recently Shaikh \cite{Shaikh} introduced
the notion of pseudo quasi-Einstein (or pseudo $\eta$-Einstein) manifolds. A SQ-Sasakian manifold $M^*$
is said to be pseudo quasi-Einstein (or pseudo $\eta$-Einstein) manifold if its Ricci tensor $S$ of type $(0,2)$
is not identically zero and satisfies the following:
\begin{equation}\label{eqn2.17}
S(X,Y)=pg(X,Y)+q\eta(X)\eta(Y)+sD(X,Y),
\end{equation}
where $p,q,s$ are scalars for which $q\neq0, s\neq0$ and $D(X,\xi)=0$ for any vector field $X$. It may be
noted that every quasi-Einstein (or $\eta$-Einstein) manifold is a pseudo quasi-Einstein (or pseudo $\eta$-Einstein)
manifold but not conversely \cite{Shaikh}.
\section{Invariant Submanifolds of SQ-Sasakian manifolds}
On the analogy of almost contact Hermitian manifolds, the invariant and anti-invariant submanifolds are
depend on the behaviour of almost contact metric structure $\phi$. A submanifold $M$ of an almost contact metric manifold
is said to be invariant \cite{BEJ} if the structure vector field $\xi$ is tangent to $M$ at every point of $M$ and $\phi X$ is tangent to
$M$ for any vector field $X$ tangent to $M$ at every point of $M$, that is $\phi(TM)\subset TM$ at every point of $M$.

Now we prove the following:
\begin{proposition}
 Let $M$ be an invariant submanifold of a SQ-Sasakian manifold $M^*$. Then the following relations hold:
\begin{equation}\label{eqn3.1}
\nabla_{X}\xi= -\beta\phi X,
\end{equation}
\begin{equation}\label{eqn3.2}
h(X,\xi)=0,
\end{equation}
\begin{equation}\label{eqn3.3}
R(X,Y)\xi=(Y\beta)\phi X-(X\beta)\phi Y + \beta^2[\eta(Y)X-\eta(X)Y].
\end{equation}
\end{proposition}
\begin{proof}
From (\ref{eqn1.4}) and (\ref{eqn2.7}) we get
\begin{eqnarray}
\label{eqn3.4}
-\beta\phi X &=&\bar{\nabla}_{X}\xi \\
\nonumber &=& \nabla_{X}\xi+h(X,\xi).
\end{eqnarray}
Since $M$ is an invariant submanifold of a SQ-Sasakian manifold $M^*$, therefore for any $X \in \Gamma(TM)$,
$\phi X\in\Gamma(TM)$. Thus equating the tangential and normal part of (\ref{eqn3.4}) we get (\ref{eqn3.1}) and (\ref{eqn3.2}).
 Also using (\ref{eqn3.1}) in (\ref{eqn2.10}), we get $R(X,Y)\xi=\bar{R}(X,Y)\xi$ and hence we get (\ref{eqn3.3}).
\end{proof}
 Recently Hui and Mandal \cite{HUI} studied Chaki-pseudo parallel contact CR-submanifolds and Deszcz-pseudo parallel
contact CR-submanifolds of Kenmotsu manifolds. Following \cite{HUI} we can define the following:
\begin{definition}\cite{HUI}
 A submanifold $M$ of a SQ-Sasakian manifold $M^*$ is called Chaki-pseudo parallel
 if its second fundamental form $h$ satisfies
\begin{equation}\label{eqn3.5}
(\nabla_{X}h)(Y,Z)=2\alpha(X)h(Y,Z)+ \alpha(Y)h(X,Z)+\alpha(Z)h(X,Y)
\end{equation}
for all $X,Y,Z$ on $M$, where $\alpha$ is a nowhere vanishing $1$-form.
\end{definition}
 In particular, if $\alpha(X)=0$ then $h$ is said to be parallel and $M$ is said to be
 parallel submanifold of $M^*$. We now prove the
  following:
\begin{theorem}
Let $M$ be an invariant submanifold of a SQ-Sasakian manifold $M^*$. Then $M$ is
totally geodesic if and only if $M$ is Chaki-pseudo parallel with $\{\alpha(\xi)\}^2+\beta^2\neq0$.
\end{theorem}
\begin{proof}
Suppose that $M$ is a Chaki-pseudo parallel invariant submanifold of a SQ-Sasakian manifold $M^*$.

Then by virtue of (\ref{eqn2.9}) we have from (\ref{eqn3.5}) that
\begin{align}\label{eqn3.6}
& \nabla^\bot_{X}h(Y,Z)-h(\nabla_{X}Y,Z)-h(Y,\nabla_{X}Z) &  \\
\nonumber & =2\alpha(X)h(Y,Z)+\alpha(Y)h(X,Z)+\alpha(Z)h(X,Y).
\end{align}
Putting $Z=\xi$ in (\ref{eqn3.6}) and using (\ref{eqn3.2}) we get
\begin{equation}\label{eqn3.7}
-h(Y,\nabla_{X}\xi)=\alpha(\xi)h(X,Y).
\end{equation}
In view of (\ref{eqn3.1}), (\ref{eqn3.7}) yields
\begin{equation}\label{eqn3.8}
\beta h(Y,\phi X)-\alpha(\xi)h(X,Y)=0.
\end{equation}
Replacing $X$ by $\phi X$ in (\ref{eqn3.8}) and using (\ref{eqn1.1}) and (\ref{eqn3.1}), we get
\begin{equation}\label{eqn3.9}
\beta h(X,Y)+\alpha(\xi)h(Y,\phi X)=0.
\end{equation}
From $(\ref{eqn3.8})$ and $(\ref{eqn3.9})$ we get $[\{\alpha(\xi)\}^2+\beta^2]h(X,Y)=0$, which implies that
$h(X,Y)=0$ for all $X,Y$ on $M$ as $[\alpha(\xi)]^2+\beta^2\neq0$. Hence $M$ is totally geodesic submanifold.

The converse part is trivial. This proves the theorem.
\end{proof}
\begin{corollary}
Let $M$ be an invariant submanifold of a SQ-Sasakian manifold $M^*$.
Then $M$ is totally geodesic if and only if $M$ is parallel.
\end{corollary}
From Theorem $3.1$, we can state the following:
\begin{corollary}
Let $M$ be an invariant submanifold of a Sasakian manifold $M^*$. Then $M$ is totally geodesic if and only if
$M$ is Chaki-pseudo parallel with $[\alpha(\xi)]^2+1\neq0$.
\end{corollary}
\begin{corollary}
  Let $M$ be an invariant submanifold of a Cosymplectic manifold $M^*$. Then $M$ is totally geodesic
if and only if $M$ is Chaki-pseudo parallel.
\end{corollary}
Again following the definition of Deszcz-pseudo Ricci symmetric manifold, we can define the following:
\begin{definition}\cite{HUI}
A submanifold $M$ of a SQ-Sasakian manifold $M^*$ is said to be Deszcz-pseudo parallel if its
second fundamental form $h$ satisfies
\begin{eqnarray}
\label{eqn3.10}
\bar{R}(X,Y)\cdot h&=&(\bar{\nabla}_X\bar{\nabla}_Y-\bar{\nabla}_Y\bar{\nabla}_X-\bar{\nabla}_{[X,Y]})h\\
\nonumber &=&L_{h}Q(g,h),
\end{eqnarray}
where $L_h$ is some function on $W=\{x \in  M:(h-Hg)_x\neq0\}$ for all vector fields $X,Y$ tangent to $M$, and $\bar{R}$
is the curvature tensor of $\bar{M}$.
\end{definition}

In particular, if $L_h=0$ then $M$ is said to be semiparallel.

We now prove the following:
\begin{theorem}
Let $M$ be an invariant submanifold of a SQ-Sasakian manifold $M^*$. Then $M$ is totally geodesic if and
only if $M$ is Deszcz-pseudo parallel with $L_h+\beta^2\neq0$
\end{theorem}
\begin{proof}
 Let $M$ be an invariant submanifold of a SQ-Sasakian manifold$M^*$. First, suppose that $M$ is
Deszcz-pseudo parallel. Then we have the relation (\ref{eqn3.10}), i.e.
\begin{equation*}
(\bar{R}(X,Y)\cdot h)(Z,U)= L_hQ(g,h)(Z,U;X,Y)
\end{equation*}
 i.e.
 \begin{align}\label{eqn3.11}
 &R^\bot(X,Y)h(Z,U)-h(R(X,Y)Z,U)-h(Z,R(X,Y)U)& \\
\nonumber =& L_{h}[g(Y,Z)h(X,U)-g(X,Z)h(Y,U)+g(Y,U)h(X,Z)- g(X,U)h(Y,Z)]
 \end{align}
 Putting $X=U=\xi$ in (\ref{eqn3.11}) and using (\ref{eqn3.2}) we get
\begin{equation}\label{eqn3.12}
h(Z,R(\xi,Y)\xi)=L_{h}h(Y,Z).
\end{equation}
 In view of (\ref{eqn3.2}) and (\ref{eqn3.3}), (\ref{eqn3.12}) yields $(L_{h}+\beta^2)h(Y,Z)=0$,
 which implies that $h(Y,Z)=0$ for all $Y,Z$ on $M$, i.e., $M$ is totally geodesic, since $L_{h}+\beta^2\neq0$.
 The converse part is trivial. This proves the theorem.
 \end{proof}
From Theorem $\ref{eqn3.2}$, we can state the following:
\begin{corollary}
Let $M$ be an invariant submanifold of a SQ-Sasakian manifold $M^*$. Then $M$ is totally
geodesic if and only if $M$ is semiparallel.
\end{corollary}
\begin{corollary}
Let $M$ be an invariant submanifold of a Sasakian manifold $M^*$. Then $M$ is totally
geodesic if and only if $M$ is Deszcz-pseudo parallel with $L_{h}+1\neq0$.
\end{corollary}
\begin{corollary}
Let $M$ be an invariant submanifold of a cosymplectic manifold of a SQ-Sasakian manifold $M^*$. Then $M$ is totally geodesic
if and only if $M$ is Deszcz-pseudo parallel.
\end{corollary}
From Corollary 3.1, Corollary 3.4, Theorem 3.1 and Theorem 3.2, we can state the following:
\begin{theorem}
Let $M$ be an invariant submanifold of a SQ-Sasakian manifold $M^*$. Then the following statements are equivalent:\\
(i) $M$ is totally geodesic,\\(ii) $M$ is parallel,\\(iii) $M$ is semiparallel,\\ (iv) $M$ is Chaki-pseudo parallel
with $\{\alpha(\xi)\}^2+\beta^2\neq0$,\\ (v) $M$ is Deszcz-pseudo parallel with $L_{h}+\beta^2\neq0$.
\end{theorem}
We now consider $M$ be an invariant submanifold of a SQ-Sasakian manifold $M^*$ with respect to semisymmetric metric
connection $\widetilde{\bar{\nabla}}$. Let $\nabla$ be the induced connection on $M$ from the connection $\bar{\nabla}$
and $\widetilde{\nabla}$ be the induced connection on $M$ from the connection $\widetilde{\bar{\nabla}}$.

Let $h$ and $\widetilde{h}$ be the second fundamental form with respect to Levi-Civita connection and semisymmetric  metric connections
respectively. Then we have:
\begin{equation}\label{eqn3.13}
\widetilde{\bar{\nabla}}_{X}Y= \widetilde{\nabla}_{X}Y + \widetilde{h}(X,Y).
\end{equation}
By virtue of (\ref{eqn2.7}) and (\ref{eqn2.12}), (\ref{eqn3.13}) yields
\begin{equation}\label{eqn3.14}
\widetilde{\nabla}_{X}Y+\widetilde{h}(X,Y)=\nabla_{X}Y+h(X,Y)+\eta(Y)X-g(X,Y)\xi.
\end{equation}
Since $X,\xi \in \Gamma(TM)$, by equating the tangential and normal components of (\ref{eqn3.14}) we get
\begin{equation}\label{eqn3.15}
\widetilde{\nabla}_{X}Y= \nabla_{X}Y+ \eta(Y)X-g(X,Y)\xi
\end{equation}
and
\begin{equation}\label{eqn3.16}
\widetilde{h}(X,Y)= h(X,Y).
\end{equation}
This leads to the following:
\begin{theorem}
Let $M$ be a submanifold of a SQ-Sasakian manifold $M^*$ with respect to semisymmetric metric connection.
Then\\(i) $M$ admits semisymmetric metric connection, \\(ii) The second fundamental forms with respect to Levi-Civita connection
and semisymmetric connection are equal.
\end{theorem}
Now we define the following:
\begin{definition}\cite{HUI}
  A submanifold $M$ of a SQ-Sasakian $M^*$ with respect to semisymmetric metric connection is called
  Chaki-pseudo parallel if its second fundamental form $\widetilde{h}$ satisfies
  \begin{equation*}
   (\widetilde{\nabla}_{X}\widetilde{h})(Y,Z)= 2\alpha(X)\widetilde{h}(Y,Z)+\alpha(Y)\widetilde{h}(X,Z)+
   \alpha(Z)\widetilde{h}(X,Y)
\end{equation*}
for all $X,Y,Z$ on $M$.
\end{definition}
Let us take $M$ be an invariant submanifold of a SQ-Sasakian manifold $M^*$ with respect to semisymmetric
metric connection. Suppose that $M$ is Chaki-pseudo parallel with respect to semisymmetric metric connection.
Then we have
\begin{equation}\label{eqn3.17}
(\widetilde{\nabla}_{X}h)(Y,Z)= 2\alpha(X)h(Y,Z)+\alpha(Y)h(X,Z)+\alpha(Z)h(X,Y).
\end{equation}
In view of (\ref{eqn3.2}) and (\ref{eqn3.15}) we have from (\ref{eqn3.17}) that
\begin{align*}
  &(\nabla_{X}h)(Y,Z)+g(h(Y,Z),\xi) -g(X,h(Y,Z))\xi& \\
  &- \eta(Y)h(X,Z)- \eta(Z)h(X,Y)& \\
   &= 2\alpha(X)h(Y,Z)+\alpha(Y)h(X,Z)+\alpha(Z)h(X,Y)
\end{align*}
i.e.,
\begin{align}\label{eqn3.18}
  &\nabla^\bot_{X}h(Y,Z)-h(\nabla_{X}Y,Z) +h(Y,\nabla_{X}Z)& \\
\nonumber  &+  g(h(Y,Z),\xi)-g(X,h(Y,Z))\xi& \\
\nonumber  &- \eta(Y)h(X,Z)-\eta(Z)h(X,Y)& \\
\nonumber  &=  2\alpha(X)h(Y,Z)+\alpha(Y)h(X,Z)+\alpha(Z)h(X,Y).
\end{align}
Putting $Z=\xi$ in (\ref{eqn3.18}) and using (\ref{eqn3.2}) we get
\begin{equation}\label{eqn3.19}
-h(Y,\nabla_{X}\xi)-h(X,Y)= \alpha(\xi)h(X,Y).
\end{equation}
By virtue of (\ref{eqn3.1}), we have from (\ref{eqn3.19}) that
\begin{equation}\label{eqn3.20}
\beta h(Y,\phi X)-\{\alpha(\xi)+1\}h(X,Y)=0.
\end{equation}
Replacing $X$ by $\phi X$ in (\ref{eqn3.20}) and using (\ref{eqn1.1}) and (\ref{eqn3.2}) we get
\begin{equation}\label{eqn3.21}
\beta h(Y,X)+\{\alpha(\xi)+1\}h(\phi X,Y)=0.
\end{equation}
From (\ref{eqn3.20}) and (\ref{eqn3.21}) we get
\begin{equation*}
[\{\alpha(\xi)+1\}^2+\beta^2]h(X,Y)=0,
\end{equation*}
which implies that $h(X,Y)=0$ provided $\{\alpha(\xi)+1\}^2+\beta^2\neq0$. Thus we can state the following:
\begin{theorem}
Let $M$ be an invariant submanifold of a SQ-Sasakian manifold $M^*$ with respect to semisymmetric metric connection.
Then $M$ is totally geodesic if and only if $M$ is Chaki-pseudo parallel with respect to semisymmetric metric connection,
provided $\{\alpha(\xi)+1\}^2+\beta^2\neq0$.
\end{theorem}
\begin{corollary}
Let $M$ be an invariant submanifold of a SQ-Sasakian manifold $M^*$ with respect to semisymmetric metric connection.
Then $M$ is totally geodesic if and only if $M$ is parallel with respect to semisymmetric connection, provided $\beta^2+1\neq0$
\end{corollary}
\begin{definition}\cite{HUI}
A submanifold $M$ of a SQ-Sasakian manifold $M^*$ with respect to semisymmetric metric connection is
said to be Deszcz-pseudo parallel with respect to semisymmetric metric connection if
\begin{equation}\label{eqn3.22}
\widetilde{\bar{R}}(X,Y)\cdot\widetilde{h}= L_{\widetilde{h}}Q(g,\widetilde{h}),
\end{equation}
where $L_{\widetilde{h}}=L_h$ is defined in (\ref{eqn3.10}) and $\widetilde{\bar{R}}$ is the curvature
tensor of $\bar{M}$ for all $X,Y \in \Gamma(TM)$. In particular, if $L_h=0$ then $M$ is said to be
semiparallel with respect to semisymmetric metric connection.
\end{definition}
We now prove the following:
\begin{theorem}
Let $M$ be an invariant submanifold of a SQ-Sasakian manifold $M^*$ with respect to semisymmetric  metric connection.
Then $M$ is totally geodesic if and only if $M$ is Deszcz-pseudo parallel with respect to semisymmetric
metric connection, provided $(L_h+\beta^2)^2+\beta^2\neq0$.
\end{theorem}
\begin{proof}
Let $M$ be an invariant submanifold of a SQ-Sasakian manifold $M^*$ with respect to semisymmetric
metric connection. Suppose that $M$ is Deszcz-pseudo parallel with respect to semisymmetric metric
connection. Then we have from (\ref{eqn3.22}) that
\begin{equation*}
 \widetilde{\bar{R}}(X,Y)\cdot h=L_hQ(g,h),
\end{equation*}
i.e.,

\begin{align}\label{eqn3.23}
   & \widetilde{R}^\bot(X,Y)h(Z,U)-h(\widetilde{R}(X,Y)Z,U)-h(Z,\widetilde{R}(X,Y)U) \\
  \nonumber & =L_h[g(Y,Z)h(X,U)-g(X,Z)h(Y,U)+g(Y,U)h(X,Z) \\
  \nonumber & -g(X,U)h(Y,Z)].
\end{align}
Putting $X=U=\xi$ in (\ref{eqn3.23}) and using (\ref{eqn3.2}), we get
\begin{equation}\label{eqn3.24}
  h(Z,\widetilde{R}(\xi,Y)\xi)=L_{h}h(Y,Z).
\end{equation}
By virtue of (\ref{eqn3.2}) and Gauss equation we have from (\ref{eqn2.14}) that
\begin{eqnarray}
\label{eqn3.25}
  \widetilde{R}(X,Y)\xi &=& \widetilde{\bar{R}}(X,Y)\xi\\
  \nonumber &=& \beta^2\{\eta(Y)X-\eta(X)Y\}+\{\beta\eta(Y)+(Y\beta)\}\phi X \\
  \nonumber &-& \{\beta\eta(X)+(X\beta)\}\phi Y.
\end{eqnarray}
In view of (\ref{eqn3.2}) and (\ref{eqn3.25}), (\ref{eqn3.24}) yields
\begin{equation}\label{eqn3.26}
  (L_h+\beta^2)h(Z,Y)+ \beta h(Z,\phi Y)=0.
\end{equation}
Replacing $Y$ by $\phi Y$ in (\ref{eqn3.26}) and using (\ref{eqn1.1}) and (\ref{eqn3.2}) we get
\begin{equation}\label{eqn3.27}
  (L_h+\beta^2)h(Z,\phi Y)- \beta h(Z,Y)=0.
\end{equation}
From (\ref{eqn3.26}) and (\ref{eqn3.27}) we get $[(L_h+\beta^2)^2+\beta^2]h(Y,Z)=0$, which implies that
$h(Y,Z)=0$, as $(L_h+\beta^2)^2+\beta^2\neq0$, i.e. $M$ is totally geodesic.
\end{proof}
\begin{corollary}
  Let $M$ be an invariant submanifold of a SQ-Sasakian manifold $M^*$ with respect to semisymmetric metric
  connection. Then $M$ is totally geodesic with respect to semisymmetric metric connection, provided $\beta^2+1 \neq 0$.
\end{corollary}
From corollary $3.7$, Corollary $3.8$, Theorem $3.5$ and Theorem $3.6$, we can state the following:
\begin{theorem}
  Let $M$ be an invariant submanifold of a SQ-Sasakian manifold $M^*$ with respect to semisymmetric connection.
  Then the following statements are equivalent:\\
  (i) $M$ is totally geodesic, \\ (ii) $M$ is parallel with respect to semisymmetric metric connection with
  $\beta^2+1 \neq 0$,\\ (iii) $M$ is semiparallel with respect to semisymmetric metric connection with
  $\beta^2+1 \neq 0$,\\ (iv) $M$ is Chaki-pseudo parallel admitting semisymmetric metric connection with
  $\{\alpha(\xi)+1\}^2+\beta^2 \neq 0$,\\ (v) $M$ is Deszcz-pseudo parallel admitting semisymmetric metric connection
  with $\{(L_h+\beta^2)^2+\beta^2 \neq 0\}$.
\end{theorem}
\maketitle
\section{Ricci Solitons on invariant submanifolds}
Let us take $(g,\xi,\lambda)$ be a Ricci soliton on an invariant submanifold $M$ of a SQ-Sasakian manifold $M^*$.
Then we have
\begin{equation}\label{eqn4.1}
  (\pounds_{\xi}g)(Y,Z)+ 2S(Y,Z)+ 2\lambda g(Y,Z)=0.
\end{equation}
From (\ref{eqn3.1}) we get
\begin{eqnarray}
\label{eqn4.2}
  (\pounds_{\xi}g)(Y,Z) &=& g(\nabla_{Y}\xi,Z)+g(Y,\nabla_{Z}\xi) \\
  \nonumber &=& -\beta[g(\phi Y,Z)+g(Y,\phi Z)]  \\
  \nonumber &=& 0.
\end{eqnarray}
Using (\ref{eqn4.2}) in (\ref{eqn4.1}) we get
\begin{equation}\label{eqn4.3}
  S(Y,Z)=-\lambda g(Y,Z).
\end{equation}
This leads to the following:
\begin{theorem}
  If $(g,\xi,\lambda)$ is a Ricci soliton on an invariant submanifold $M$ of a SQ-Sasakian manifold $M^*$
  then $M$ is Einstein.
\end{theorem}
From (\ref{eqn3.3}) we get
\begin{equation*}
  S(Y,\xi)=-((\phi Y)\beta)+ 2m\beta^2\eta(Y)
\end{equation*}
and hence
\begin{equation}\label{eqn4.4}
  S(\xi,\xi)=2m\beta^2.
\end{equation}
Also from (\ref{eqn4.3}) we get
\begin{equation}\label{eqn4.5}
  S(\xi,\xi)=-\lambda.
\end{equation}
Thus from (\ref{eqn4.4}) and (\ref{eqn4.5}) we obtain $\lambda= -2m\beta^2<0$. This leads to the following:
\begin{theorem}
  A Ricci soliton $(g,\xi,\lambda)$ on an invariant submanifold $M$ of a SQ-Sasakian manifold $M^*$ is always
  shrinking.
\end{theorem}
Now we take $(g,\xi,\lambda)$ is a Ricci soliton on an invariant submanifold $M$ of a SQ-Sasakian manifold $M^*$
with respect to semisymmetric metric connection $\widetilde{\bar{\nabla}}$. Then we have
\begin{equation}\label{eqn4.6}
  (\widetilde{\pounds}_{\xi}g)(Y,Z)+ 2\widetilde{S}(Y,Z)+ 2\lambda g(Y,Z)=0.
\end{equation}
From (\ref{eqn3.15}) we get
\begin{equation*}
  \widetilde{\nabla}_{X}\xi=-\beta\phi X + X - \eta(X)\xi
\end{equation*}
and hence

\begin{eqnarray}
\label{eqn4.7}
 (\widetilde{\pounds}_{\xi}g)(Y,Z) &=& g(\widetilde{\nabla}_{Y}\xi,Z)+ g(Y,\widetilde{\nabla}_{Z}\xi) \\
\nonumber &=& 2[g(Y,Z)-\eta(Y)\eta(Z)].
\end{eqnarray}
Also for submanifold $M$ with respect to induced semisymmetric metric connection, we can calculate
\begin{equation}\label{eqn4.8}
  \widetilde{S}(Y,Z)= S(Y,Z)- (2m-1)\{g(Y,Z)-\eta(Y)\eta(Z)+\Phi(Y,Z)\}.
\end{equation}
By virtue of (\ref{eqn4.7}) and (\ref{eqn4.8}), (\ref{eqn4.6}) yields
\begin{equation}\label{eqn4.9}
  S(Y,Z)=(2m-\lambda+2)g(Y,Z)+ 2m\eta(Y)\eta(Z)+(2m-1)\Phi(Y,Z),
\end{equation}
which implies that $M$ is pseudo $\eta$-Einstein. This leads to the following:
\begin{theorem}
  If $(g,\xi,\lambda)$ is a Ricci soliton on an invariant submanifold $M$ of a SQ-Sasakian manifold $M^*$
  with respect to semisymmetric metric connection, then $M$ is pseudo $\eta$-Einstein.
\end{theorem}
\maketitle
\section{Anti-invariant submanifolds of SQ-Sasakian manifolds and Ricci Solitons}
A submanifold $M$ of a SQ-Sasakian manifold $M^*$ is said to be anti-invariant \cite{BEJ} if for any $X$ tangent
to $M$, $\phi X$ is normal to $M$, i.e. $\phi(TM)\subset T^\bot M$ at every point of $M$, where $T^\bot M$ is the normal bundle of $M$.

Let us take $M$ be an anti-invariant submanifold of a SQ-Sasakian manifold $M^*$ such that $\xi$ is tangent to $M$. Then the relation
(\ref{eqn3.4}) holds. Since $M$ is anti-invariant, so $\phi X \in T^\bot M$ for any $X \in TM$. So, equating the tangential and normal
components of (\ref{eqn3.4}) we obtain $\nabla_X\xi=0$ and $h(X,\xi)= -\beta\phi X$. This leads to the followng:
\begin{proposition}
  Let $M$ be an anti-invariant submanifold of a SQ-Sasakian manifold $M^*$ such that $\xi$ is tangent to $M$. Then the following
  relations hold:
  \begin{equation}\label{eqn5.1}
    \nabla_X\xi = 0,
  \end{equation}
\begin{equation}\label{eqn5.2}
h(X,\xi)= -\beta \phi X.
\end{equation}
\end{proposition}
Next, we prove the following:
\begin{proposition}
  Let $M$ be an anti-invariant submanifold of a SQ-Sasakian manifold $M^*$ such that $\xi$ is normal to $M$. Then the following relations hold:
\begin{equation}\label{eqn5.3}
A_\xi X = 0,
\end{equation}
\begin{equation}\label{eqn5.4}
\nabla^\bot_X \xi= \phi X
\end{equation}
for all $X \in TM$.
\end{proposition}
\begin{proof}
For any $X \in TM$ and $\xi \in T^\bot M$, we have from (\ref{eqn1.4}) and (\ref{eqn2.8}) that
\begin{eqnarray}
\label{eqn5.5}
  -\beta\phi X &=& \bar{\nabla}_X\xi \\
  \nonumber &=& -A_\xi X+ \nabla^\bot_X\xi.
\end{eqnarray}
Since $M$ is an anti-invariant submanfold of $M^*$, therefore $\phi X \in T^\bot M$ for any $X \in TM$. Hence equating the tangential and normal
components from both side of (\ref{eqn5.5}) we get (\ref{eqn5.3}) and (\ref{eqn5.4}).
\end{proof}
We now consider $M$ be an anti-invariant submanifold of a SQ-Sasakian manifold $M^*$ with respect to semisymmetric metric connection $\widetilde{\bar{\nabla}}$ such that $\xi$ is tangent to $M$. Then for any $X \in TM$, we have from (\ref{eqn1.4}), (\ref{eqn2.12}) and
(\ref{eqn3.13}) that
\begin{eqnarray}
\label{eqn5.6}
 \widetilde{\nabla}_X\xi + \widetilde{h}(X,\xi)&=& \widetilde{\bar{\nabla}}_X\xi \\
 \nonumber &=& \bar{\nabla}_X\xi+X-\eta(X)\xi \\
 \nonumber &=& -\beta\phi X+X-\eta (X)\xi,
\end{eqnarray}
where $\widetilde{h}$ is the second fundamental form of $M$ with respect to semisymmetric metric connection. Since
$M$ is an anti-invariant submanifold of $M^*$, equating the tangential and normal components of (\ref{eqn5.6}) we can state the following:
\begin{proposition}
Let $M$ be an anti-invariant submanifold of a SQ-Sasakian manifold $M^*$ with respect to semisymmetric connection such that $\xi$
 is tangent to $M$. Then the following relations hold:
\begin{equation}\label{eqn5.7}
\widetilde{\nabla}_X\xi=X-\eta(X)\xi,
\end{equation}
\begin{equation}\label{eqn5.8}
\widetilde{h}(X,\xi)= -\beta \phi X
\end{equation}
for any $X \in TM $.
\end{proposition}
If $M$ is an anti-invariant submanifold of a SQ-Sasakian manifold $M^*$ with respect to semisymmetric metric connection such that
$\xi$ is normal to $M$. Then for any $X \in TM$ we have from (\ref{eqn1.4}), (\ref{eqn2.8}) and (\ref{eqn2.12}) that
\begin{eqnarray*}
  -\widetilde{A}_\xi X+\widetilde{\nabla}^\bot_X \xi &=& \widetilde{\bar{\nabla}}_X \xi \\
  \nonumber &=& \bar{\nabla}_X \xi+X-\eta(X)\xi\\
  \nonumber &=& -\beta\phi X+X,
\end{eqnarray*}
from which equating the tangential and normal components of above equation, we can state the following:
\begin{proposition}
  Let $M$ be an anti-invariant submanifold of a SQ-Sasakian manifold $M^*$ with respect to semisymmetric metric connection such that $\xi$
  is normal to $M$. Then the following relations hold:
 \begin{equation}\label{eqn5.9}
 \widetilde{A}_\xi X= -X,
 \end{equation}
 \begin{equation}\label{eqn5.10}
 \widetilde{\nabla}^\bot_X\xi = -\beta \phi X
 \end{equation}
  for any $X \in TM $.

\end{proposition}
 Let us take $(g,\xi,\lambda)$ be a Ricci soliton on an anti-invariant submanifold $M$ of a SQ-Sasakian manifold $M^*$
  such that $\xi$ is tangent to $M$. Then we have the  relation (\ref{eqn4.1}). From (\ref{eqn5.1}), we can obtain
 \begin{equation}\label{eqn5.11}
  (\pounds_\xi g)(Y,Z)=0.
  \end{equation}
 Using (\ref{eqn5.11}) in (\ref{eqn4.1}) we get the relation (\ref{eqn4.3}) and hence we can state the following:
 \begin{theorem}
   If $(g,\xi,\lambda)$ is a Ricci Soliton on an anti-invariant submanifold $M$ of a SQ-Sasakian manifold $M^*$ such that $\xi$
   is tangent to $M$, then $M$ is Einstein.
 \end{theorem}
 Again, if $(g,\xi,\lambda)$ is a Ricci Soliton on an anti-invariant submanifold $M$ of a SQ-Sasakian manifold $M^*$ with
 respect to semisymmetric metric connection such that $\xi$ is tangent to $M$. Then we have the relation (\ref{eqn4.6}).
 From (\ref{eqn5.7}) we can calculate
 \begin{equation}\label{eqn5.12}
   (\widetilde{\pounds}_\xi g)(Y,Z)= 2[g(Y,Z)-\eta(Y)\eta(Z)].
 \end{equation}
So, using (\ref{eqn4.8}) and (\ref{eqn5.12}) in (\ref{eqn4.6}), we get the relation (\ref{eqn4.9}), which implies that $M$
is pseudo $\eta$-Einstein. Thus we can state the following:
\begin{theorem}
  If $(g.\xi,\lambda)$ is Ricci Soliton on an anti-invariant submanifold $M$ of an SQ-Sasakian manifold $M^*$ with respect to
  semisymmetric metric connection such that $\xi$ is tangent to $M$, then $M$ is pseudo $\eta$-Einstein.
\end{theorem}

\maketitle
\section{Conclusion}
In this paper, Chaki-pseudo parallel and Deszcz-pseudo parallel invariant submanifolds of SQ-Sasakian manifolds are studied.
It is known that Chaki-pseudo Ricci symmetric manifolds and Deszcz-pseudo Ricci symmetric manifolds are different. However,
it is proved that Chaki-pseudo parallel invariant submanifolds of SQ-Sasakian manifolds and Deszcz-pseudo parallel invariant submanifolds
of SQ-Sasakian manifolds are equivalent with a certain condition. Also it is shown that Chaki-pseudo parallel invariant submanifolds
of SQ-Sasakian manifolds with respect to semisymmetric metric connection and Deszcz-pseudo parallel invariant submanifolds of SQ-Sasakian
manifolds with respect to semisymmetric metric connection are equivalent with respect to a certain condition.

Among others, we have studied invariant and anti-invariant submanifolds of SQ-Sasakian manifolds $M^*$ whose metric are Ricci solitons. From Theorem $4.1$, Theorem $4.3$, Theorem $5.1$ and Theorem $5.2$, we can state the following:
\begin{theorem}
  Let $(g,\xi,\lambda)$ be a Ricci soliton on a submanifold $M$ of a SQ-Sasakian manifold $M^*$.
  Then the following holds:\\

\begin{tabular}{|c|c|c|}
  \hline
  \textbf{nature of submanifold} & \textbf{connection of $M^*$} & \textbf{M}  \\  \cline{1-3}
  invariant submanifold & Levi-Civita & Einstein \\ \cline{2-3}
  &semisymmetric metric&pseudo $\eta$-Einstein \\ \cline{1-3}

  anti-invariant submanifold& Levi-Civita & Einstein \\ \cline{2-3}
  such that $\xi$ is tangent to $M$&semisymmetric metric &pseudo $\eta$-Einstein \\ \cline{1-3}
  \hline
\end{tabular}

\end{theorem}
\noindent\textbf{Acknowledgement:} The authors gratefully acknowledges to the SERB (Project No: EMR/2015/002302), Govt. of India for financial assistance of the work.

\vspace{0.1in}
\noindent Shyamal Kumar Hui and Joydeb Roy\\
Department of Mathematics, The University of Burdwan,\\
Burdwan, 713104, West Bengal, India.\\
E-mail: skhui@math.buruniv.ac.in; joydeb.roy8@gmail.com

\end{document}